\documentclass{aptpub}
\numberwithin{equation}{section}
\makeatletter
\renewcommand{\crnotice}{} 
\makeatother
\usepackage[colorlinks = true,linkcolor = blue,urlcolor  = blue,citecolor = blue,anchorcolor = blue]{hyperref}

\authornames{S.-R. Hsiau and Y.-S. Lin} 
\shorttitle{A Double Secretary Problem} 

\begin{document}

\title{A Double Secretary Problem} 

\authorone[Department of Mathematics, National Changhua University of Education]{Shoou-Ren Hsiau} 
\authortwo[Department of Mathematics, Victoria Academy]{Yi-Shen Lin} 

\addressone{No. 1, Jin-De Road., Changhua
500, Taiwan, R.O.C.} 
\emailone{srhsiau@cc.ncue.edu.tw} 
\addresstwo{No. 1110, Zhengnan Road, Doulin City, Yunlin County 640, Taiwan, R.O.C.}
\emailtwo{ychyslin@163.com}

\begin{abstract}
We consider a double secretary problem which contains $2n$ applicants of $n$ different qualities, two of each quality. As in the classical secretary problem (CSP),
the applicants are interviewed sequentially in a random order by a manager and
the manager wants to find an optimal stopping rule that maximizes the probability of selecting one of the two best applicants. We prove that the problem
leads to a monotone structure and so the corresponding one-stage look-ahead
rule (which is of threshold type) is optimal. The limiting behavior of the rule and of
the maximum probability are studied and compared with that of the CSP.
\end{abstract}

\keywords{monotone stopping rules; one-stage look-ahead; optimal stopping; secretary problem.}

\ams{60G40}{62L15}    

\section{Introduction} \label{sec1} 
A classical sequential selection problem is the so-called secretary (best-choice) problem
(abbreviated as CSP), which can be stated as follows. Suppose a manager of a company
wants to hire a secretary and knows that there are $n$ (fixed) applicants applying sequentially
in a random order for the job. Assume that they can be classified in a unique order by the manager from the best, $1$, to the worst, $n$. After interviewing an applicant, the manager is only able to observe the applicant's relative rank among those that have been interviewed so far, and the manager must decide either to hire the present applicant (and stop the decision problem) or to reject the applicant and interview further applicants. No recall of preceding applicants is permitted. The optimal stopping rule which maximizes the probability of selecting the best of the $n$ applicants is well known (cf.\ Lindley \cite{LindleyDV1961}) to reject the first $a_n-1$ applicants and then accept the next relatively best applicant, where
\begin{align*}
a_n \;=\; \min\left\{\, i \ge 1 : \sum_{j=i+1}^{n} \frac{1}{j-1} \le 1 \,\right\},
\end{align*}
with the convention that $\sum_{j=\ell}^{m} x_j := 0$ whenever $\ell>m$. Moreover, $a_n/n \to 1/e$ as $n\to\infty$ and the probability of selecting the best applicant using the optimal stopping rule is
\begin{align*}
\frac{a_n-1}{n}\sum_{j=a_n}^{n}\frac{1}{j-1},
\end{align*}
which also converges to $1/e$ as $n\to\infty$.

The literature with respect to the variation of the CSP is vast. Nonetheless, we mention a few topics with some references for the interested reader. Gilbert and Mosteller \cite{Gilbert1966} and Sakaguchi \cite{Sakaguchi1978} studied the model of allowing the manager to choose more than one of the $n$ applicants to obtain the best one. Mucci \cite{Mucci1973a, Mucci1973b} considered an extension of the CSP to general nondecreasing payoff functions of the selected applicant's absolute rank. Chow et al. \cite{chow1964optimal} studied the problem of minimizing the expected rank of the selected applicant. In Tamaki \cite{Tamaki1979}, the manager is allowed to have two choices and the selections are considered success if both the best and the second best applicants are selected. Later, Hsiau and Yang \cite{Hsiau2000} considered the problem with group interviews in which the manager succeeds whenever the selected group contains the best applicant. Szajowski \cite{Szajowski1982} and Lin et al. \cite{Lin2019} considered and solved the problem of optimally selecting the $k$-th best applicant. Another variation introduced by Presman and Sonin \cite{Presman1972} assumes that the (total) number of applicants is a positive integer-valued random variable instead of being a fixed known integer $n$. Bruss \cite{bruss1984unified} proposed a unified approach model in which no assumption of (the distribution of) $n$ is made and the intervention of time is allowed. See Ferguson \cite{Ferguson1989}, Freeman \cite{FreemanPR1983} and Samuels \cite{Samuels1991} for further references. 

The model considered in this paper is described as follows. There are $2n$ ($n$ is known) applicants applying for the secretarial position. Assume the manager can rank the applicants from $1$ (best) to $n$ (worst) with each (absolute) rank being duplicate, i.e.\ the absolute ranks of them are $1,1,2,2,\ldots,n,n$. They are interviewed sequentially in a random order with each of the $(2n)!/2^{n}$ orderings being equally likely. Upon interviewing an applicant, the manager has to decide either to accept the present applicant for the position or to reject the applicant and continue interviewing the next one. During the interview, the manager can rank the applicant among all the preceding ones and aware of whether the applicant's capabilities of being a secretary appears the second time (i.e. the applicant's absolute rank reappears). For example, suppose the absolute ranks of the first four applicants are $3,2,4,2$. When the fourth applicant is being interviewed, the manager knows that the relative rank is $1$ and selecting the fourth applicant is equivalent to selecting the second one. No recall is allowed. The object is to maximize the probability of selecting one of the two best of the $2n$ applicants. In Garrod, Kubicki and Morayne \cite{garrod2012choose}, and Grau Ribas \cite{ribas2019new}, the same problem is addressed in different methods from ours.

The rest of this paper is organized as follows. In Section \ref{sec2}, by adopting the method used in Dynkin \cite{DynkinEB1963}, we prove that the problem leads to a monotone case optimal stopping problem (cf. Chow and Robbins \cite{chow1961martingale} and Ferguson \cite[Chapter 5]{Ferguson2008}) and so the optimal stopping rule is the one-stage look-ahead rule, which is of threshold type with threshold $r_n$. It is shown that the optimal rule stops on the first relatively best applicant whose absolute rank has reappeared and the number of \emph{distinct} absolute ranks among those that have been interviewed is larger than or equal to $r_n$ (cf. Theorem \ref{thm 2.1}). Moreover, we show that $r := \lim_{n\to\infty} r_n/n$ exists and satisfies the equation $r e^{-2/r} = e^{-5}$ (implying that $r \approx 0.470927$). In addition, the maximum probability $p_n$ of selecting one of the two best applicants is obtained and its limiting value is
\[
r+\frac{4}{3r}\left[(1-r)^{3/2}-(1-r)^2\right]\approx 0.767974.
\]
Section \ref{sec3} contains comparisons between our problem and the CSP along with concluding remarks. The proof of some technical materials presented in Section \ref{sec2} is put in Section \ref{sec4}.

\section{The Monotone Case Problem}\label{sec2}
In our model, the decisions are based on the sequential information stemming from the history of relative ranks and the number of distinct absolute ranks. To study this problem, let $R_j$, $j=1,2,\ldots,2n$, be the absolute rank of the $j$-th applicant such that, for every permutation $\sigma_{2n}$ of $(1,1,\ldots,n,n)$,
\begin{equation} \label{2.1}
P\bigl((R_1,\ldots,R_{2n})=\sigma_{2n}\bigr)=\frac{2^{n}}{(2n)!}.    
\end{equation}
For $j=1,2,\ldots,2n$, define $X_j$ to be the number of different values in $\{R_i: R_i\leq R_j,\,1\leq i\leq j\}$, i.e. the relative rank of the $j$-th applicant among the first $j$ applicants. Up to the $j$-th interview, define $D_j$ to be the number of distinct values in $\{R_1,R_2,\ldots,R_j\}$ and $S_j$ indicates the number of times that the relatively best applicant has appeared (once or twice). Let $X'_j$, $j=1,2,\ldots,n$, be the relative rank of the $j$-th \emph{new number}. More precisely, for $j=1,2,\ldots,n$, letting $k_j=\min\{i:\,D_i=j,\,1\leq i\leq 2n\}$, we have
\[
X'_j=X_{k_j}.
\]
For example, if the absolute ranks of the first six applicants are $2,3,1,1,3,4$, then $X_1=1$, $X_2=2$, $X_3=1$, $X_4=1$, $X_5=3$, $X_6=4$, $X'_1=1$, $X'_2=2$, $X'_3=1$ and $X'_4=4$; furthermore, the configurations of the $D$'s and $S$'s are $(D_1,S_1)=(1,1)$, $(D_2,S_2)=(2,1)$, $(D_3,S_3)=(3,1)$, $(D_4,S_4)=(3,2)$, $(D_5,S_5)=(3,2)$ and $(D_6,S_6)=(4,2)$. Since $(R_1,\ldots,R_{2n})$ is exchangeable, it is readily seen that $X'_1,X'_2,\ldots,X'_n$ are independent with $X'_j$ being uniformly distributed on $\{1,2,\ldots,j\}$ (cf.\ Rényi \cite{renyi1962theorie}). We want to find a stopping rule $\tau_n\in\mathcal{M}_n$ such that
\begin{equation} \label{2.2}
P(R_{\tau_n}=1)=\sup_{\tau\in\mathcal{M}_n} P(R_\tau=1),   \end{equation}
where $\mathcal{M}_n$ denotes the set of all finite stopping rules adapted to the filtration $\{\mathcal{F}_j\}$, $\mathcal{F}_j$ being the $\sigma$-algebra generated by $(X_1,D_1), (X_2,D_2), \ldots, (X_j,D_j)$.

Note that an applicant would be selected only if it is relatively best among those that have been interviewed. For $j=1,2,\ldots,2n$, let $T_j$ be the time at which the $j$-th relatively best applicant appears, i.e.\ $T_1=1$ and, for $j=2,3,\ldots,2n$, define $T_j=\min\{i:i>T_{j-1},\,X_i=1\}$ (with $\min\varnothing := 2n$). Set $\sum_{i=\ell}^{m} c_{\ell}=0$ whenever $\ell>m$, and, for $i=1,\ldots,n$, define
\begin{equation} \label{2.3}
\alpha(i)=\frac{2n}{i}+\sum_{j=i+1}^{n}\frac{1}{j-1}.     
\end{equation}
Let
\begin{equation}\label{2.4}
r_n=\min\{\, i\ge 1:\ \alpha(i)\le 5 \,\}.    
\end{equation}
The following theorem is our main result.

\begin{thm}\label{thm 2.1}The stopping rule
\[
\tau_n=\min\{\, i\ge 1:\ X_i=1,\ S_i=2\ \text{and}\ D_i\ge r_n \,\}
\]
is optimal for the problem \eqref{2.2}, i.e.\ the optimal rule stops on the first $T_j$ with $D_{T_j}\ge r_n$ and $S_{T_j}=2$.
\end{thm}

Before proving Theorem \ref{thm 2.1}, we note some general facts concerning the evolution of the number of distinct absolute ranks and that of the relatively best applicant has appeared up to an interview. Specifically, suppose at some stage $j$ there are $i$ distinct absolute ranks that have been observed so far, and the relatively best applicant has appeared $m$ times ($m=1,2$). Let $(i,m)$ denote this configuration, i.e. $(D_j,S_j)=(i,m)$. Recall that $X'_1,X'_2,\ldots,X'_n$ are independent with $X'_j$ being uniformly distributed on $\{1,2,\ldots,j\}$. This implies that $\{(D_i,S_i)\}_{i=1,2,\ldots,2n}$ is a Markov chain and the transition probailities are as follows:
\begin{equation}\label{2.5}
\begin{cases}
& P_{j,j+1}((i,1),(i,2))=\frac{1}{2n-j},\\ 
& P_{j,j+1}((i,1),(i,1))=\frac{2i-j-1}{2n-j}, \\
& P_{j,j+1}((i,1),(i+1,1))=\frac{2(n-i)}{2n-j}
\end{cases}
\end{equation}

\begin{equation}\label{2.6}
\begin{cases}
& P_{j,j+1}((i,2),(i,2))=\frac{2i-j}{2n-j},\\ 
& P_{j,j+1}((i,2),(i+1,1))=\frac{2(n-i)}{2n-j}\cdot\frac{1}{i+1}, \\
& P_{j,j+1}((i,2),(i+1,2))=\frac{2(n-i)}{2n-j}\cdot\frac{i}{i+1}.
\end{cases}
\end{equation}
Let $T$ be the time at which the last relatively best applicant appears, i.e. the second one of the two best applicants appears.

For $j=1,\ldots,2n-1$ put $j'=\min\{k>j:\ X_k=1 \text{ and } S_k=2\}$, and we define, for $i=1,2,\ldots,n$,
\begin{equation} \label{2.7}
A_{i,j}=P\bigl(T=j' \mid D_{j}=i,\, S_{j}=1\bigr),\qquad 2i>j\geq i
\end{equation}
and
\begin{equation} \label{2.8}
B_{i,j}=P\bigl(T=j' \mid D_{j}=i,\, S_{j}=2\bigr), \qquad 2i\geq j>i.
\end{equation}
(To be clear, $\min\varnothing := 2n$.)

To prove Theorem \ref{thm 2.1}, we shall make use of the following lemma, whose proof is relegated to Section \ref{sec4}.

\begin{lem} \label{lem 2.1}
   Define $A_{i,j}$ and $B_{i,j}$ as in \eqref{2.7} and \eqref{2.8}. Then
    \[
A_{i,j}=\frac{2n+i}{3n}:=A_i,\qquad 2i>j\geq i
\]
\textit{and}
\[
B_{i,j}=\frac{2(n-i)}{3n}+\frac{i}{3n}\sum_{\ell=i+1}^{n}\frac{1}{\ell-1}:=B_i,\qquad 2i\geq j>i.
\]
\end{lem}
\begin{proof}[Proof of Theorem \ref{thm 2.1}]
Note that the optimal rule evidently stops on a relatively best applicant whose absolute rank has appeared twice. 

Set $\tilde{T}_0=1$, and for $j=1,2,\ldots,n$, let $\tilde{T}_j=\left(\tilde{T}_{j-1}\right)'=\min\{k>\tilde{T}_{j-1}:\ X_k=1 \text{ and } S_k=2\}$ ( $\min\varnothing := 2n$). Now we restrict our attention on the times $\tilde{T}_j$'s and let
\[
Z_j=P\bigl(T=\tilde{T}_j \mid D_{\tilde{T}_j}\bigr), \quad j=1,\ldots,n.
\]
Then our original optimal stopping problem \eqref{2.2} is reduced to that for the process $\{Z_j,\mathcal{F}_{\tilde T_j}\}_{j=1,\ldots,n}$. More precisely, letting $\widetilde{\mathcal{M}}_n $ denote the class of all finite stopping rules adapted to $\{\mathcal{F}_{\tilde T_j}\}$, we want to find a stopping rule $\zeta_n\in\widetilde{\mathcal{M}}_n$ such that
\begin{equation} \label{2.9}
E(Z_{\zeta_n})=\sup_{\zeta\in\widetilde{\mathcal{M}}_n}E(Z_\zeta).
\end{equation}

We claim that the problem \eqref{2.9} is monotone, i.e.\ the sequence of events $\{Z_j\ge E(Z_{j+1}\mid \mathcal{F}_{\tilde T_j})\}$ is increasing in $j$. Suppose on some $\{\tilde T_j<2n\}$ we have $D_{\tilde T_j}=i$ for some $i=1,\ldots,n$. Then
\[
\begin{aligned}
Z_j
&=P\!\left(T=\tilde T_j \mid D_{\tilde T_j}=i\right)\\
&=P\!\left(T=\tilde T_j \mid D_{\tilde T_j}=i,\ S_{\tilde T_j}=2\right)\\
&=P\!\left(X'_k\neq 1 \,\text{for all} \,k\geq i+1\right)\\
&=\prod_{\ell=i}^{n-1}\frac{\ell}{\ell+1}
=\frac{i}{n}.
\end{aligned}
\]

If we continue from $D_{\tilde T_j}=i$ and stop at $\tilde T_{j+1}$, we expect to gain
\[
\begin{aligned}
E(Z_{j+1}\mid \mathcal{F}_{\tilde T_j})
&=P\!\left(T=\tilde T_{j+1}\mid D_{\tilde T_j}=i\right)\\
&=P\!\left(T=\tilde T_{j+1}\mid D_{\tilde T_j}=i,\ S_{\tilde T_j}=2\right)\\
&=B_i,
\end{aligned}
\]
where $B_i=\dfrac{2(n-i)}{3n}+\dfrac{i}{3n}\sum_{\ell=i+1}^{n}\dfrac{1}{\ell-1}$ by Lemma \ref{lem 2.1}. It follows that $Z_j\ge E(Z_{j+1}\mid \mathcal{F}_{\tilde T_j})$, if and only if
\[
\frac{D_{\tilde T_j}}{n}\ \ge\ B_{D_{\tilde T_j}}
=\frac{2(n-D_{\tilde T_j})}{3n}+\frac{D_{\tilde T_j}}{3n}\sum_{\ell=D_{\tilde T_j}+1}^{n}\frac{1}{\ell-1},
\]
which is equivalent to $\alpha(D_{\tilde T_j})\le 5$, where $\alpha(i)$ is defined in \eqref{2.3}. Hence, $\alpha(D_{\tilde T_{j+1}})\le \alpha(D_{\tilde T_j})\le 5$ (since $D_{\tilde T_j}$ is increasing in $j$ and $\alpha(i)$ is decreasing in $i$), which in turn implies that $Z_{j+1}\ge E(Z_{j+2}\mid \mathcal{F}_{\tilde T_{j+1}})$. As a result, $\{Z_j\ge E(Z_{j+1}\mid \mathcal{F}_{\tilde T_j})\}\subset \{Z_{j+1}\ge E(Z_{j+2}\mid \mathcal{F}_{\tilde T_{j+1}})\}$, and hence the problem \eqref{2.9} is a (finite horizon) monotone stopping problem. Therefore, the one-stage look-ahead rule is optimal (cf.\ Chow and Robbins \cite{chow1961martingale} and Ferguson \cite[Chapter 5]{Ferguson2008}), i.e.,
\begin{equation} \label{2.10}
\begin{aligned}
\zeta_n
&=\min\bigl\{\, j\ge 1:\ Z_j\ge E(Z_{j+1}\mid \mathcal{F}_{\tilde T_j}) \,\bigr\}\\
&=\min\bigl\{\, j\ge 1:\ \alpha(D_{\tilde T_j})\le 5 \,\bigr\}\\
&=\min\bigl\{\, j\ge 1:\ D_{\tilde T_j}\ge r_n \,\bigr\}.
\end{aligned}
\end{equation}
is optimal for \eqref{2.9}, where the third equality follows from \eqref{2.3}-\eqref{2.4} and the monotonicity of $D_{T_j}$ and $\alpha(i)$. Hence, our original optimal stopping problem \eqref{2.2} has an optimal stopping rule
\begin{align*}
\tau_n &=\min\left\{\tilde{T}_j:D_{\tilde{T}_j}\geq r_n\right\}\\
&= \min\{T_j : D_{T_j} \ge r_n \text{ and } S_{T_j} = 2\}\\
      &= \min\{i \ge 1 : X_i = 1, D_i \ge r_n \text{ and } S_i = 2\}.
\end{align*}
The proof is complete.
\end{proof}

\begin{thm}\label{thm 2.2}
Let $f(x)=x e^{-2/x}$, $x>0$. We have that $r:=\lim_{n\to\infty} r_n/n$ exists and satisfies $f(r)=e^{-5}$. Note that $r$ is approximately $0.470927$.
\end{thm}

\begin{proof}
Note that $r_2=1$, $r_3=2$ and $r_n$ is increasing in $n$. Thus, $r_n\ge 2$ for $n\ge 3$. By definition (cf.\ \eqref{2.3}–\eqref{2.4}),
\begin{equation}\label{2.11}
\begin{aligned} 
5 \;<\; \alpha(r_n-1)
&= \frac{2n}{r_n-1} + \sum_{j=r_n}^{n}\frac{1}{j-1} \\
&= \frac{2n+1}{r_n-1} + \sum_{j=r_n}^{n-1}\frac{1}{j} \\
&< \frac{2n+1}{r_n-1} + \int_{\,r_n-\frac{1}{2}}^{\,n-\frac{1}{2}} \frac{dx}{x} \\
&= \frac{2n+1}{r_n-1} + \log\!\left(\frac{n-\frac{1}{2}}{\,r_n-\frac{1}{2}\,}\right) \\
&=  \frac{2+\frac{1}{n}}{\frac{r_n}{n}-\frac{1}{n}}
   + \log\!\left(\frac{1-\frac{1}{2n}}{\frac{r_n}{n}-\frac{1}{2n}}\right). 
\end{aligned}
\end{equation}
Letting $r'_n=r_n/n$, it follows that
\begin{equation} \label{2.12}
\Bigl(r'_n-\frac{1}{2n}\Bigr)\,
e^{-\frac{2+\frac{1}{n}}{r'_n-\frac{1}{n}}}
\;<\;
\Bigl(1-\frac{1}{2n}\Bigr)\,e^{-5}.
\end{equation}
On the other hand, we have
\begin{equation} \label{2.13}
\begin{aligned}
5 \;\ge\; \alpha(r_n)
&= \frac{2n}{r_n} + \sum_{j=r_n+1}^{n}\frac{1}{j-1} \\
&= \frac{2n+1}{r_n} + \sum_{j=r_n+1}^{n-1}\frac{1}{j}\\
&> \frac{2n+1}{r_n} \;+\; \int_{r_n+1}^{n} \frac{dx}{x} \\
&= \frac{2n+1}{r_n} \;+\; \log\!\left(\frac{n}{r_n+1}\right) \\
&= \frac{2+\frac{1}{n}}{r'_n} \;+\; \log\!\left(\frac{1}{\,r'_n+\frac{1}{n}\,}\right),
\end{aligned}
\end{equation}
from which it follows that
\begin{equation} \label{2.14}
\left(r'_n+\frac{1}{n}\right) \, e^{-\frac{2+\frac{1}{n}}{r'_n}} \;>\; e^{-5}.
\end{equation}
Since $2/n \le r'_n \le 1$ for each $n \ge 3$, $\{r'_n\}_{n\ge3}$ is a bounded sequence. For each accumulation point $s$ of $\{r'_n\}_{n\ge3}$, it follows from \eqref{2.12} and \eqref{2.14} that $f(s)=e^{-5}$.

By \eqref{2.14},
\[
r'_n+\frac{1}{n} \;>\; e^{-5}
\]
implying that $s \ge e^{-5} > 0$. In addition, elementary calculus yields that $f(x)$ is strictly increasing and continuous in $(0,\infty)$ with $\lim_{x\to0+} f(x)=0$ and $\lim_{x\to\infty} f(x)=\infty$. Hence, $f(x)=e^{-5}$ has a unique solution in $(0,\infty)$ and this implies $\{r'_n\}$ has only one accumulation point $r$. It is clear that $r=\lim_{n\rightarrow \infty} r_n$ and satisfies $f(r)=e^{-5}$. The proof is complete.
\end{proof}

\begin{thm} \label{thm 2.3}
    The maximum probability of selecting one of the two best applicants is
    \[
\frac{1}{3n}\left\{ \left(1-r_n+\sum_{i=1}^{r_n-1}\prod_{j=i}^{r_n-1}\frac{2(n-j)}{2(n-j)+1}\right)\left(3-\sum_{\ell=r_n}^{n-1}\frac{1}{\ell}\right)+2n+r_n\right\},
\]
and
\[
\lim\limits_{n\rightarrow \infty}p_n=r+\frac{4}{3r}\left[(1-r)^{3/2}-(1-r)^2\right]=0.767974\ldots.
\]
\end{thm}

\begin{proof}
Recall that 
\[
k_j=\min\{1\leq i\leq 2n:D_j=i\}
\quad\text{and}\quad j'=\min\{k>j:X_k=1\,\text{and}\,S_k=2\}.
\]
Let $\eta_n=k_{r_n}$. By the definitions of $p_n$ and $\tau_n$, we have
\begin{equation} \label{2.15 eq}
    \begin{aligned}
        p_n&=P(R_{\tau_n}=1)\\
        &=P(R_{\tau_n}=1 \,\mid\,S_{\eta_n}=1)P(S_{\eta_n}=1)\\
        &\quad+P(R_{\tau_n}=1 \,\mid\,S_{\eta_n}=2)P(S_{\eta_n}=2)\\
        &=P(T=\eta'_n\,\mid\, S_{\eta_n}=1)P(S_{\eta_n}=1)\\
        &\quad+P(T=\eta'_n\,\mid\,S_{\eta_n}=2)P(S_{\eta_n}=2).
    \end{aligned}
\end{equation}
Note that, by Lemma~\ref{lem 2.1}, 
\begin{align*}
    P(T=\eta_n'\,\mid\,S_{\eta_n}=1)&=P(T=\eta'_n\,\mid\, D_{\eta_n}=r_n,\,S_{\eta_n}=1)\\
    &=A_{r_n}
\end{align*}
and
\begin{align*}
    P(T=\eta_n'\,\mid\,S_{\eta_n}=2)&=P(T=\eta'_n\,\mid\, D_{\eta_n}=r_n,\,S_{\eta_n}=2)\\
    &=B_{r_n}.
\end{align*}
We need to find $P(S_{\eta_n}=1)$, and note that $P(S_{\eta_n}=2)=1-P(S_{\eta_n}=1)$. By the definition of $\eta_n$, we have 
\begin{align*}
    P(S_{\eta_n}=1)=P(S_{\eta_n}=1,\,X'_{r_n}=1)+P(S_{\eta_n}=1,\,X'_{r_n}>1).
\end{align*}
Note that 
\begin{align*}
    P(S_{\eta_n}=1,\,X'_{r_n}=1)=P(X'_{r_n}=1)=\frac{1}{r_n}.
\end{align*}
Moreover, 
\begin{equation*}
    \begin{aligned}
        &P(S_{\eta_n}=1,\,X'_{r_n}>1)\\
        =&\sum_{i=1}^{r_n-1}P(X'_i=1,\,X_s>1\,\text{for all}\,\, k_{r_n}\geq s>k_i)\\
        =&\sum_{i=1}^{r_n-1}P(X'_i=1,\,X_s>1\,\text{for all}\,\, k_{i+1}\geq s>k_i,\\
        &\qquad\quad\, X_s>1\,\text{for all}\,\, k_{i+2}\geq s>k_{i+1},\ldots,\\
        &\qquad\quad\,X_s>1\,\text{for all}\,\, k_{r_n}\geq s>k_{r_n-1})\\
        =&\sum_{i=1}^{r_n-1}\frac{1}{i}\prod_{j=i}^{r_n-1}\left(\frac{2(n-j)}{2(n-j)+1}\cdot\frac{j}{j+1}\right)\\
        =&\frac{1}{r_n}\sum_{i=1}^{r_n-1}\prod_{j=i}^{r_n-1}\frac{2(n-j)}{2(n-j)+1}.
    \end{aligned}
\end{equation*}
So, 
\begin{align*}
    P(S_{\eta_n}=1)=\frac{1}{r_n}\left(1+\sum_{i=1}^{r_n-1}\prod_{j=i}^{r_n-1}\frac{2(n-j)}{2(n-j)+1}\right).
\end{align*}
Now \eqref{2.15 eq} becomes, by Lemma~\ref{lem 2.1},
\begin{equation*}
\begin{aligned}
    p_n&=A_{r_n}\cdot P(S_{\eta_n}=1)+B_{r_n}\cdot(1-P(S_{\eta_n}=1))\\
    &=P(S_{\eta_n}=1)(A_{r_n}-B_{r_n})+B_{r_n}\\
    &=\frac{1}{r_n}\left(1+\sum_{i=1}^{r_n-1}\prod_{j=i}^{r_n-1}\frac{2(n-j)}{2(n-j)+1}\right)\cdot\\
    &\quad \frac{r_n}{3n}\left(3-\sum_{\ell=r_n}^{n-1}\frac{1}{\ell}\right)+\frac{2(n-r_n)}{3n}+\frac{r_n}{3n}\sum_{\ell=r_n}^{n-1}\frac{1}{\ell}\\
    &=\frac{1}{3n}\left\{\left(1-r_n+\sum_{i=1}^{r_n-1}\prod_{j=i}^{r_n-1}\frac{2(n-j)}{2(n-j)+1}\right)\left(3-\sum_{\ell=r_n}^{n-1}\frac{1}{\ell}\right)+2n+r_n\right\}.
\end{aligned}
\end{equation*}
Using the facts that $e^{-x}=1-x+o(x^2)$, and $\sum_{\ell=1}^N\frac{1}{\ell}=\ln N+c_N$ with $c_N\rightarrow \gamma$, it is not difficult to prove that 
\begin{align*}
    \lim\limits_{n\rightarrow \infty}p_n=r+\frac{4}{3r}\left[(1-r)^{3/2}-(1-r)^2\right]=0.767974\ldots.
\end{align*}
The proof is complete.
\end{proof}

\section{Comparing with the CSP and Concluding Remarks}\label{sec3}

Recall that $X'_j$, $j=1,\ldots,n$, stands for the relative rank of the $j$-th new number. As alluded to in Section \ref{sec1}, $X'_1,\ldots,X'_n$ are independent with $X'_j$ being uniformly distributed in $\{1,2,\ldots,j\}$. Let $R'_j$, $j=1,\ldots,n$, be the absolute rank of the $j$-th new number. Let $\mathcal{M}'_n$ denote the set of all finite stopping rules adapted to the filtration $\mathcal{F}'_j$, where $\mathcal{F}'_j$ is the $\sigma$-algebra generated by $X'_1,\ldots,X'_j$. The optimal stopping problem of finding $\nu_n \in \mathcal{M}'_n$ such that
\[
P(R'_{\nu_n}=1)=\sup_{\tau\in\mathcal{M}'_n} P(R'_{\tau}=1)
\]
is equivalent to the CSP. Therefore, $\nu_n=\min\{a_n \le j \le n:\ X'_j=1\}$, where
\[
a_n=\min\left\{\, i\ge 1:\ \sum_{j=i+1}^{n}\frac{1}{j-1}\le 1 \,\right\},
\]
and the optimal probability is
\[
q_n=\frac{a_n-1}{n}\sum_{j=a_n}^{n}\frac{1}{j-1}.
\]
Recall that for $1\leq j\leq n$, define $k_j=\min\{i:D_i=j,\,1\leq i\leq 2n\}$. Now $k_{\nu_n}\in \mathcal{M}_n$, and $P(R_{k_{\nu_n}}=1)=q_n$. Hence $p_n\geq q_n$. In fact, we have the stronger relation: $\{R_{k_{\nu_n}}=1\}\subset \{R_{\tau_n}=1\}$. For $n\geq 3$, if the sequence of the absolute ranks of the $2n$ applicants is $s=(1,s_2,s_3,\ldots,s_{2n-2},1)$, then $s\in\{R_{\tau_n}=1\}$, but $s\notin\{R_{k_{\nu_n}}=1\}$ since $a_n\geq 2$ for $n\geq 3$. This implies that $p_n>q_n$, for $n\geq 3$. Note that $p_1=q_1=1$, $p_2=5/6$, $q_2=1/2$. Hence $p_n>q_n$ for $n\geq 2$. We have proved the part $(i)$ in Theorem~\ref{thm 3.1} below. 

\begin{thm} \label{thm 3.1}
\textup{(i)} For each $n=2,3,\ldots$, $p_n>q_n$.\\
\textup{(ii)} For $8\geq n\geq 1$, $r_n=a_n$ except that $r_7=4$, $a_7=3$; for $n\geq 9$, $r_n>a_n$.
\end{thm}
\begin{proof}[Proof of Theorem~\ref{thm 3.1} (ii)]
By computing $a_n$ and $r_n$ for $8\geq n \geq 1$, we see that 
\begin{align*}
    a_{2i-1}=a_{2i}=r_{2i-1}=r_{2i}=i
\end{align*}
for $i=1,2,3,4$, except that $a_7=3$ and $r_7=4$.

It is known that $\dfrac{n}{e} < a_n < \dfrac{\,n-\frac12\,}{e}+\dfrac{3}{2}$ (cf.\ Gilbert and Mosteller \cite{Gilbert1966}) and $\dfrac{n-\frac12}{e}+\dfrac{3}{2}\leq \dfrac{n}{2}$ for $n\geq 10$. Therefore, $a_n\leq \dfrac{n-1}{2}$ for $n\geq 10$. When $n\geq 10$, we have 
\begin{align*}
    \alpha(a_n)&=\frac{2n}{a_n}+\sum_{\ell=a_n+1}^n\frac{1}{\ell-1}\\
    &=\frac{2n-2}{a_n}+\frac{2}{a_n}+\sum_{\ell=a_n+1}^n\frac{1}{\ell-1}\\
    &\geq \frac{2n-2}{a_n}+\sum_{\ell=a_n}^n\frac{1}{\ell-1}\\
    &>4+1=5,
\end{align*}
by the definition of $a_n$ and the fact that when $n\geq 10$, we have $a_n\geq 2$ and so $\dfrac{2}{a_n}\geq \dfrac{1}{a_n-1}$. Since $\alpha(a_n)>5$ for $n\geq 10$, we see that $r_n>a_n$ for $n\geq 10$, by the definition of $r_n$. Moreover, it is easy to derive that $a_9=4$ and $r_9=5$. Hence, $r_n>a_n$ for $n\geq 9$. 

The proof is complete.
    
\end{proof}

\section{Proof of Lemma \ref{lem 2.1}} \label{sec4}

\begin{proof}[Proof of Lemma \ref{lem 2.1}]
We first use backward induction on $i$ and $j$ to prove the formula for $A_{i,j}$. Recall that for $j=1,2,\ldots,2n-1, \,j'=\min\{k>j:X_k=1,\,S_k=2\}$. When $i=n$ and $j=2n-1$, 
\begin{align*}
A_{n,2n-1}&=P(T=(2n-1)'\mid D_{2n-1}=n,\, S_{2n-1}=1)\\
&=P_{2n-1,2n}((n,1),(n,2))\\
&=1=\frac{2n+n}{3n}.
\end{align*}
In fact, it is clear that for $2n>j\geq n$,
\begin{align*}
    A_{n,j}&=P(T=j'\mid D_j=n,\,S_j=1)\\
    &=1=\frac{2n+n}{3n}.
\end{align*}
Suppose that for some $i+1\leq n$ and all $2(i+1)>j\geq i+1$, $A_{i+1,j}=(2n+i+1)/(3n)$ hold. Then, by \eqref{2.5},
\begin{align*}
    A_{i,2i-1}=&P(T=(2i-1)'\mid D_{2i-1}=i,\, S_{2i-1}=1)\\
=&P_{2i-1,2i}((i,1),(i,2))\cdot P(X'_j\neq 1\,\text{for all}\, j\geq i+1) \\
&+P_{2i-1,2i}((i,1),(i,1))\cdot P(T=(2i)'\mid D_{2i}=i,\,S_{2i}=1)\\
&+P_{2i-1,2i}((i,1),(i+1,1))\cdot P(T=(2i)'\mid D_{2i}=i+1,\,S_{2i}=1)\\
=&\frac{1}{2n-(2i-1)}\prod_{\ell=i}^{n-1}\frac{\ell}{\ell+1}+0+\frac{2(n-i)}{2n-(2i-1)}\cdot A_{i+1,2i}\\
=&\frac{1}{2n-2i+1}\cdot \frac{i}{n}+\frac{2(n-i)}{2n-2i+1}\cdot \frac{2n+i+1}{3n}\\
=&\frac{2n+i}{3n}.
\end{align*}
Furthermore, suppose $A_{i,j+1}=(2n+i)/(3n)$ holds for some $i\leq j\leq 2i-2$. Then, by \eqref{2.5},
\begin{align*}
    A_{i,j}=&P(T=j'\mid D_j=i,\, S_j=1)\\
=&P_{j,j+1}((i,1),(i,2))\cdot P(X'_j\neq 1\,\text{for all}\, j\geq i+1)\\
&+P_{j,j+1}((i,1),(i,1))\cdot P(T=(j+1)'\mid D_{j+1}=i,\,S_{j+1}=1)\\
&+P_{j,j+1}((i,1),(i+1,1))\cdot P(T=(j+1)'\mid D_{j+1}=i+1,\,S_{j+1}=1)\\
=&\frac{1}{2n-j}\prod_{\ell=i}^{n-1}\frac{\ell}{\ell+1}+\frac{2i-j-1}{2n-j}\cdot A_{i,j+1}+\frac{2(n-i)}{2n-j}\cdot A_{i+1,j+1}\\
=&\frac{1}{2n-j}\cdot \frac{i}{n}+\frac{2i-j-1}{2n-j}\cdot \frac{2n+i}{3n}+\frac{2(n-i)}{2n-j}\cdot\frac{2n+i+1}{3n}\\
=&\frac{2n+i}{3n}.
\end{align*}
The above arguments complete the proof for the formula $A_{i,j}=(2n+i)/(3n)$, $2i>j\geq i$. Next, we again use the backward induction on $i$ and $j$ to prove the formula for $B_{i,j}$. When $i=n$, and $2n\geq j>n$, it is clear that 
\begin{align*} 
    B_{n,j}&=P(T=j'\mid D_j=n,\, S_j=2)\\
    &=0\\
    &=\frac{2(n-n)}{3n}+\frac{n}{3n}\sum_{\ell=n+1}^n\frac{1}{\ell-1}.
\end{align*}
Suppose that for some $i+1\leq n$ and all $2(i+1)\geq j>i+1$, 
\begin{align*}
    B_{i+1,j}=\frac{2(n-i-1)}{3n}+\frac{i+1}{3n}\sum_{\ell=i+2}^{n}\frac{1}{\ell-1}
\end{align*}
hold. Then, by \eqref{2.6}, 
\begin{align*}
    B_{i,2i}=&P(T=(2i)'\mid D_{2i}=i,\,S_{2i}=2)\\
    =&P_{2i,2i+1}((i,2),(i,2))\cdot P(T=(2i+1)'\mid D_{2i+1}=i,\,S_{2i+1}=2)\\
    &+P_{2i,2i+1}((i,2),(i+1,1))\cdot P(T=(2i+1)'\mid D_{2i+1}=i+1,\,S_{2i+1}=1)\\
    &+P_{2i,2i+1}((i,2),(i+1,2))\cdot P(T=(2i+1)'\mid D_{2i+1}=i+1,\,S_{2i+1}=2)\\
    =&0+\frac{2(n-i)}{2n-2i}\cdot\frac{1}{i+1}\cdot A_{i+1,2i+1}+\frac{2(n-i)}{2n-2i}\cdot\frac{i}{i+1}\cdot B_{i+1,2i+1}\\
    =&\frac{1}{i+1}\cdot\frac{2n+i+1}{3n}+\frac{i}{i+1}\cdot\left\{ \frac{2(n-i-1)}{3n}+\frac{i+1}{3n}\sum_{\ell=i+2}^n\frac{1}{\ell-1}\right\}\\
    =&\frac{2(n-i)}{3n}+\frac{i}{3n}\sum_{\ell=i+1}^n\frac{1}{\ell-1}.
\end{align*}
Furthermore, suppose, for some $2i-1\geq j>i-1$, 
\begin{align*}
B_{i,j+1}=\frac{2(n-i)}{3n}+\frac{i}{3n}\sum_{\ell=i+1}^n\frac{1}{\ell-1}    
\end{align*}
holds. Then, by \eqref{2.6},
\begin{align*}
    B_{i,j}=&P(T=j'\mid D_j=i,\,S_j=2)\\
    =&P_{j,j+1}((i,2),(i,2))\cdot P(T=(j+1)'\mid D_{j+1}=i,\,S_{j+1}=2)\\
    &+P_{j,j+1}((i,2),(i+1,1))\cdot P(T=(j+1)'\mid D_{j+1}=i+1,\,S_{j+1}=1)\\
    &+P_{j,j+1}((i,2),(i+1,2))\cdot P(T=(j+1)'\mid D_{j+1}=i+1,\,S_{j+1}=2)\\
    =&\frac{2i-j}{2n-j}\cdot B_{i,j+1}+\frac{2(n-i)}{2n-j}\cdot\frac{1}{i+1}\cdot A_{i+1,j+1}+\frac{2(n-i)}{2n-j}\cdot\frac{i}{i+1}\cdot B_{i+1,j+1}\\
    =&\frac{2i-j}{2n-j}\cdot\left\{\frac{2(n-i)}{3n}+\frac{i}{3n}\sum_{\ell=i+1}^n\frac{1}{\ell-1}\right\}+\frac{2(n-i)}{2n-j}\cdot \frac{1}{i+1}\cdot\frac{2n+i+1}{3n}\\
    &+\frac{2(n-i)}{2n-j}\cdot\frac{i}{i+1}\cdot\left\{\frac{2(n-i-1)}{3n}+\frac{i+1}{3n}\sum_{\ell=i+2}^n\frac{1}{\ell-1}\right\}\\
    =&\frac{2i-j}{2n-j}\cdot\left\{\frac{2(n-i)}{3n}+\frac{i}{3n}\sum_{\ell=i+1}^n\frac{1}{\ell-1}\right\}+\frac{2(n-i)}{2n-j}\cdot \left\{\frac{2(n-i)}{3n}+\frac{i}{3n}\sum_{\ell=i+1}^n\frac{1}{\ell-1}\right\}\\
    =&\frac{2(n-i)}{3n}+\frac{i}{3n}\sum_{\ell=i+1}^n\frac{1}{\ell-1}.
\end{align*}
The above arguments complete the proof for the formula 
\begin{align*}
    B_{i,j}=\frac{2(n-i)}{3n}+\frac{i}{3n}\sum_{\ell=i+1}^n \frac{1}{\ell-1},\qquad 2i\geq j>i. 
\end{align*}
The proof is complete.
\end{proof}




%
%
%


\begin{thebibliography}{99}
\bibliographystyle{APT}
\footnotesize

\bibitem{bruss1984unified}
{\sc Bruss, F.T.}
(1984).  A unified approach to a class of best choice problems with an unknown number of options. {\em Ann. Probab.} {\bf{12}}, 882--889.

\bibitem{chow1961martingale}
{\sc Chow, Y.-S. and Robbins, H}
(1961).  A martingale system theorem and applications. {\em  Proc. 4th Berkeley Sympos. Math. Statist. and Prob.} {\bf{1}}, 93--104.

\bibitem{chow1964optimal}
{\sc Chow, Y.-S., Moriguti, S., Robbins, H. and Samuels, S.M.}
(1964). Optimal selection based on relative rank. {\em  Israel J. Math.} {\bf{2}}, 81--90.

\bibitem{DynkinEB1963}
{\sc  Dynkin, E.B.}
(1963). Optimal choice of the stopping moment of a Markov process. {\em  Dokl. Akad. Nauk SSSR} {\bf{150}}, 238--240 (in Russian).

\bibitem{Ferguson1989}
{\sc  Ferguson, T.S.}
(1989).  Who solved the secretary problem? {\em   Statistical Science.} {\bf{4}}, 282--296.

\bibitem{Ferguson2008}
{\sc  Ferguson, T.S.}
(2008).  Optimal Stopping and Applications. {\em  Mathematics Department, UCLA.} http://www.math.ucla.edu/ tom/Stopping/Contents.html.

\bibitem{FreemanPR1983}
{\sc Freeman, P. R. }
(1983).  The secretary problem and its extensions: a review. {\em  Int. Statist. Rev.} {\bf{51}}, 189--206

\bibitem{garrod2012choose}
{\sc Garrod, B., Kubicki, G., and Morayne, M.}
(2012).  How to choose the best twins. {\em  SIAM Journal on Discrete Mathematics} {\bf{26}}, 384--398

\bibitem{ribas2019new}
{\sc Grau Ribas, J.M.}
(2012).  A new look at the returning secretary problem. {\em  Journal of Combinatorial Optimization} {\bf{37}}, 1216--1236

\bibitem{Gilbert1966}
{\sc  Gilbert, J. and Mosteller, F.}
(1966).   Recognizing the maximum of a sequence. {\em  J. Amer. Statist. Assoc.} {\bf{61}}, 35--73.

\bibitem{Hsiau2000}
{\sc  Hsiau, S.-R. and Yang, J.-R.}
(2000). A natural variation of the standard secretary problem. {\em Statistica Sinica.} {\bf{10}}, 639--646.

\bibitem{LindleyDV1961}
{\sc  Lindley, D.V.}
(1961). Dynamic programming and decision theory. {\em Appl. Statist.} {\bf{10}}, 39--51.

\bibitem{Lin2019}
{\sc Lin, Y.-S., Hsiau, S.-R. and Yao, Y.-C.}
(2019). Optimal selection of the k-th best candidate. {\em  Probab. Engrg. Inform. Sci.} {\bf{33}}, 327--347.

\bibitem{Mucci1973a}
{\sc Mucci, A.G.}
(1973). Differential equations and optimal choice problems.  {\em  Ann. Statist.} {\bf{1}}, 104--113.

\bibitem{Mucci1973b}
{\sc Mucci, A.G.}
(1973). On a class of secretary problems. {\em Ann. Probab. } {\bf{1}}, 417--427.

\bibitem{Presman1972}
{\sc Presman, E. L. and Sonin, I. M}
(1972). The best choice problem for a random number of objects. {\em  Theory Prob. Appl.} {\bf{17}}, 657--668.

\bibitem{renyi1962theorie}
{\sc R{\'e}nyi, A.}
(1962). Th{\'e}ories des {\'e}l{\'e}ments saillants d'une suite d'observations. {\em  Proc., Colloq. Comb.
 Meth. in Prob. Theory (Aarhus Universitet).} {\bf{17}}, 104--115.

\bibitem{Sakaguchi1978}
{\sc Sakaguchi, M.}
(1978). Dowry problems and OLA policies. {\em  Rep. Statist. Appl. Res. Un. Japan. Sci.
 Engrs.} {\bf{25}}, 124--128.

\bibitem{Samuels1991}
{\sc Samuels, S.M.}
(1991).  Secretary problems. In Handbook of Sequential Analysis (Statist. Textbooks
 Monogr. 118), eds B. K. Ghosh and P.K. Sen, Marcel Dekker, {\em  New York.} 381--405.

\bibitem{Szajowski1982}
{\sc Szajowski, K. }
(1982). Optimal choice problem of a-th object. {\em   Mat. Stos.} {\bf{19}}, 51--65 (in Polish).

\bibitem{Tamaki1979}
{\sc Tamaki, M.}
(1979). Recognizing both the maximum and the second maximum of a sequence. {\em  J. Appl.
 Prob.} {\bf{16}}, 803--812.


\end{thebibliography}
\end{document}